\documentclass[11 pt]{article}
\usepackage{amsmath, amssymb}
\usepackage{amsthm}

\theoremstyle{plain} 
\newtheorem{theorem}{Theorem}[section]

\newtheorem{corollary}[theorem]{\indent\sc Corollary}
\newtheorem{proposition}[theorem]{Proposition}

\theoremstyle{plain} 
\newtheorem{definition}[theorem]{Definition}
\newtheorem{remark}[theorem]{Remark}

\numberwithin{equation}{section}

\makeatother

\title{\uppercase {On The Prolongations of Homogeneous Vector Bundles}}

\author{ H\"{u}lya Kad{\i}o\u{g}lu \\ \small Department of Mathematics Education, Yildiz Technical University\\ \small hkadio@yildiz.edu.tr}
\date{}
\begin{document}
   \maketitle
   
    \footnote{ 
2000 \textit{Mathematics Subject Classification}.
Primary 53C30; Secondary 55R91.
}
\footnote{ 
\textit{Key words and phrases:} Homogeneous Space, Fiber bundles, Prolongation, Homogeneous Vector Bundles
}
   \begin{abstract} In this paper, we introduce a study of prolongations of homogeneous vector bundles. We give an alternative approach for the prolongation. For a given homogeneous vector bundle $E$, we obtain a new homogeneous vector bundle. The homogeneous structure and its corresponding representation are derived. The prolongation of induced representation, which is an infinite dimensional linear representation, is also defined.
   
\end{abstract}
 
    \section{Introduction}

\hspace{5mm} In this study, we continue to work on prolongations. In our previous work \cite{Myarticle},  we have defined prolongations of finite-dimensional real representations of Lie groups and obtained faithful representations on tangent bundles of  Lie groups \cite{Myarticle}. In this work, we use the prolongations of these representations to give an alternative method to prolonge a vector bundle, specially a homogeneous vector bundle, which also has group actions in its structure. We also give the definition of the prolongation of induced representations. In the literature, the well known method for prolongation is to use lifts( for example vertical lifts or complete lifts)\cite{Yano} or to use jet prolongations \cite{Saunders}. For example in \cite{Fisher}, Fisher and Laguer worked on the second order tangent bundles by using jets. For further information about jet manifolds, we refer to \cite{Cordero},\cite{Saunders}.

Homogeneous vector bundles were studied, because of their applications to cohomologies and complex analytic Lie groups. In 1957, Raoul Bott \cite{Bott} dealt with induced representations in the framework of complex analytic Lie groups. In 1964, Griffiths gave differential-geometric
derivations of various properties of homogeneous complex manifolds. He gave some differential geometry applications to homogeneous
vector bundles and to the study of sheaf cohomology \cite{Griffiths}.  In 1988, Purohit \cite{Purohit} showed that there is a one-to-one correspondence between homogeneous vector bundles and linear representations. Various other studies about homogeneous vector bundles can be found in the literature.(\cite{Boralevi},\cite{Harboush})  

Moreover, in 1972, {\it{R. W. Brockett and H. J. Sussmann}} described how the tangent bundle of a homogeneous space can be viewed as a homogeneous space \cite{Brockett}. They associated every Lie group $G$ with another Lie group $G^*=Lie(G)\times G$ constructed as a semi direct product with the group operation given by 
\begin{equation}
(a,g).(a',g')=(a+ad(g)(a'),gg')
\end{equation}

\noindent where $(a,g), (a',g') \in G^*$. They also showed that if $G$ acts on a manifold $X$, then there exists a left action of $G^*$ on $TX$ by 
\begin{equation}
(a,g).v= d\sigma_g (v)+\bar{a}(g.\pi(v))\hspace{2mm} for\hspace{1mm} all\hspace{1mm} v\in TX
\end{equation}  
\noindent Here $\pi$ denotes the natural projection from $TX$ onto $X$ (i.e. $\pi(v)=x$ if and only if $v\in X$) and $\sigma_g:X \to X$ is the map $ x \to gx$. Clearly, both $\sigma_g(v)$ and $\bar{a}(g.\pi(v))$ belong to $T_{g.\pi(v)}$, so the sum is well-defined.
   
This paper is organized as follows. In section \ref{pre} we give some basic definitions and theorems that we need for our proofs. In section \ref{Pro}, we give homogeneous vector bundle structure of the prolonged bundle. And at the end , in section \ref{conc}, we give the conclusion and future work.

\section{Preliminaries}\label{pre}

First of all, we give the definition of a homogeneous vector bundle.

\begin{definition}
Let $G$ be a Lie group, $F$ be a $n$ dimensional real vector space, and $G$ acts transitively on a manifold $M$. Let $H$ be the isotropy subgroup of $G$ at a fixed point $p_0 \in M$ so that $M$ becomes the coset space $G/H$. In addition, suppose $G$ acts on the vector bundle $E$ sitting over $G/H$ so that its action on the base agrees with the usual action of $G$ an cosets. Then such a structure $(E, \pi, M, F)$ is called a homogeneous vector bundle.\cite{Purohit}
\end{definition}

\begin{theorem}\label{sigmaofE}

Homogeneous vector bundles over $G/H$ are in one-to-one correspondence with linear representations of $H$ \cite{Purohit}.
\end{theorem}
Above mentioned representation is defined as follows:

Let $(E, \pi, M, F)$ be a homogeneous vector bundle where $M=G/H$, $H$ be the isotropy subgroup of $G$ at $p_0$, and $F=\pi^{-1}(p_0)$. Then, there exists a Lie group representation \\$\sigma:H \to Aut(F)$ with
\begin{equation}
\begin{gathered}
\sigma(h):F \to F\\
\hspace{3cm} q \to \sigma(h)(q)=hq
\end{gathered}
\end{equation}
where $h \in H$. 

Conversely, if $G$ is a Lie group with isotropy subgroup $H$, $F$ is a finite dimensional real vector space and $\sigma:H \to Aut(F)$ is a representation, then there exists a homogeneous vector bundle $(E,\pi,G/H,F)$ where $E=G\times_H F$ \cite{Adams, Kobayashi}.

Homogeneous vector bundles also corresponds infinite dimensional real representations. This correspondence is illustrated in the following proposition.    

\begin{proposition}\cite{Purohit}
Let $(E, \pi,M, F)$ be a homogeneous vector bundle, where $M=G/H$, and $\Gamma(E)$ denotes (global) cross sections of the vector bundle $E$. For all $g \in G$, $\rho(g)$ can be defined by the following:
\begin{equation}
\begin{gathered}
\rho(g):\Gamma(E) \to \Gamma(E)\hspace{12cm}\\
\hspace{3cm} \psi \to \rho(g)(\psi):M \to E \hspace{12cm}\\
 p \to (\rho(g)(\psi))(p)=g.\psi(g^{-1}p)\hspace{25mm}
\end{gathered}
\end{equation}
Clearly, $\rho$ is a representation of $G$ in $\Gamma(E)$ which is induced by the representation $\sigma$ that is defined in theorem \ref{sigmaofE}. We call the representation $\rho$ as the induced representation of $E$.   
\end{proposition}

\begin{theorem} If a Lie group $G$ acts transitively and with maximal rank on a
differentiable manifold $X$, then $G^*$ acts transitively and with maximal rank
on the tangent bundle of $X$ \cite{Brockett}.
\end{theorem}
We have following remarks for above theorem:
\begin{remark}

 \begin{enumerate}
 \item Clearly, above result implies that the tangent bundle of a coset space
$G/H$ is again a coset space and moreover, is of the form $G^*/K$ for some
closed subgroup $K$ of $G^*$. 
\item If $H$ is a closed subgroup of $G$, then $H^*$ can be identified, in an
obvious way, with a closed subgroup of $G^*$. One verifies easily that the
isotropy group of $0_x$ corresponding to the action of $G^*$ on $TX$ is
precisely $H_x^*$, where $H_x$ is the isotropy group of $x$ corresponding to the
action of $G$ on $X$ In particular, we have the diffeomorphism $T(G/H)\simeq G^*/H^*$.
\end{enumerate}
\end{remark}

\begin{definition}
Let $\Phi$ be a $n$-dimensional Lie group representation on $G$. "The prolongation of the representation $\Phi$", which is denoted by $\widetilde{\Phi}$ is given by the following equation:
\begin{eqnarray}
\widetilde{\Phi}:TG \to GL(2n)\hspace{54mm} \nonumber \\
                 (a,v) \to \widetilde{\Phi}(a,v)=\
                 \begin{pmatrix}
                 \Phi(a) & 0\\
                 [(d(\Phi)_e(v))_j^i].\Phi(a) & \Phi(a)
                 \end{pmatrix}.
\end{eqnarray}
\end{definition}

\section {\bf{Prolongation of a Homogeneous Vector Bundle}} \label{Pro}

In this section, we'll introduce a homogeneous vector bundle with prolonged Lie group representation, which is defined in \cite{Myarticle}. 
\newline

\noindent{\bf{Main Result:}} \\

Let $(E,\pi, G|H, F)$ be the homogeneous vector bundle with the corresponding Lie group representation $\sigma: H\to Aut(F)$. Using the prolongation of the representation $\sigma$ defined in \cite{Myarticle}, we have 
\begin{equation}
\tilde{\sigma}(dR_h(a))=\
\begin{pmatrix}
\sigma(h) & 0\\
d(\sigma)_e(a).\sigma(h) & \sigma(h)
\end{pmatrix}
\end{equation}

where $\tilde{\sigma}: TH \to Aut(TF)$ is the prolongation of representation $\sigma$. 
Getting composition of $\Theta$ and $\tilde{\sigma}$, we define $\sigma^*:H^* \to Aut(TF)$ as follows:

\begin{equation}
\sigma^*(a,h)=\
\begin{pmatrix}
\sigma(h) & 0\\
d(\sigma)_e(a).\sigma(h) & \sigma(h)
\end{pmatrix}
\end{equation}

where $\Theta$ denotes the natural diffeomorphism $\Theta: H^* \to TH$. 
In the following configuration, we give the summary of what implications we will have next.
\begin{equation}
E \longleftrightarrow \sigma\longrightarrow \sigma^* \longleftrightarrow E^*
\end{equation}
where $E=(E, \pi, G/H, F)$,  $\sigma:H \to Aut(F)$, $\sigma^*:H^* \to Aut(TF)$ and $ E^*=(E^*, \pi^*, G^*/H^*, TF)$.
Now we define the new induced action which will be used for defining equivalence classes.
\newline
 
\subsection{\bf{The Prolonged Action of $H^*$ on $G^*\times TF$:}}

Using the above representation $\sigma^*$, and the natural action of $G^*$ to its coset space $G^*/H^*$, we have the following induced action $\tilde{\alpha}:(G^* \times TF)\times H^* \to G^* \times TF$ which can be obtained by the coset space definition:

\begin{equation}
 \tilde{\alpha}(a,g,v,b,h)=((a,g).(b,h),\sigma^*((b,h)^{-1})(v)).\label{complex}
\end{equation}
  
\begin{proposition}
If $v=(\xi ,u)\in TF$ and $\sigma^*((b,h)^{-1})(v)=(\tilde{\xi},\tilde{u})$, then
\begin{eqnarray}
 \tilde{\alpha}(a,g,v,b,h)=(a+adj(g)(b), gh,\tilde{\xi},\tilde{u}).
\end{eqnarray}

\end{proposition}

\begin{proof}
 Since $(a,g).(b,h)=(a+adj(g)(b),gh)$ and $(b,h)^{-1}=(adj(h^{-1})(-b),h^{-1})$, we have 
\begin{equation}
\sigma^*((b,h)^{-1})(v)=\sigma^*(adj(h^{-1})(-b),h^{-1})(\xi,u)\nonumber
\end{equation}
Using the definition of $\sigma^*$ we have
\begin{equation}
                       =\
                       \begin{pmatrix}
                       \sigma(h^{-1}) & 0 \\
                       d(\sigma)_e(-adj(h^{-1})(b)) & \sigma(h^{-1}) 
                       \end{pmatrix}.\
                       \begin{pmatrix}
                       \xi \\
                       u
                       \end{pmatrix}.\nonumber 
\end{equation}   

\begin{equation}                            
                       =(\sigma(h^{-1})(\xi), [(d\sigma)_e(-adj(h^{-1})(b))]_j^i \xi_ix_j +[\sigma(h^{-1})]_j^i u_i \dot{x}_j)\label{action}
\end{equation}
\begin{equation}
                       =(\tilde{\xi},\tilde{u})\label{sigma}
\end{equation}   
where $v=(\xi,u) \in TF$ for all $(a,g) \in G^*$ and $(b,h) \in H^*$.
Therefore, using equation (\ref{sigma}), we finish the proof.
\end{proof}

 Using above action, it is possible to form an equivalence relation on $G^* \times TF$.
\newline

\subsection{\bf{The Prolonged Equivalence Relation:}} 

If we use above induced action, we have the following equivalence relation on $G^* \times TF$ by \cite{Adams, Kobayashi}:
\newline

$((a,g),(\xi, u)) \simeq((a',g'),(\xi', u'))$ if and only if there exists $(b,h) \in H^*$ such that the following equation holds
\begin{equation}
((a',g'),(\xi', u'))=((a,g),(\xi, u)).(b,h).\label{equiv}
\end{equation}

The next corollary gives the simplified form of equivalence relation that we have defined above.
\begin{corollary}
The equivalence relation defined by the equation (\ref{equiv}) can be given by the following.
\begin{equation}
((a,g),(\xi, u))\simeq ((a',g'),(\xi', u')) \Leftrightarrow  \left\{ \begin{array}{rcl} a'=a+adj(g)(b),\hspace{6cm}\\ g'=gh,\hspace{78mm} \\ \xi'=\sigma(h^{-1})(\xi),\hspace{66mm} \\ u'=[(d\sigma)_e(-adj(h^{-1})(b))]_j^i \xi_i x_j+[\sigma(h^{-1})]_j^i u_i \dot{x}_j.\hspace{12mm}\end{array}\right. \label{equi}
\end{equation}
\end{corollary}

\begin{proof}
If $((a,g),(\xi, u))\simeq ((a',g'),(\xi', u'))$, then there exist $(b,h) \in H^*$ such that 
\begin{eqnarray}
((a',g'),(\xi', u'))&=&((a,g),(\xi, u)).(b,h).\nonumber\\
                    &=&\sigma^*(a,g,(\xi,u),b,h)
\end{eqnarray}
Using equation (\ref{action}), we have $\xi'=\sigma(h^{-1})(\xi)$ and $u'=[(d\sigma)_e(-adj(h^{-1})(b))]_j^i \xi_ix_j +[\sigma(h^{-1})]_j^i u_i \dot{x}_j. $

Moreover, since $\sigma^*(a,g,(\xi,u),b,h)\in T_{(a,g).(b,h)}$, then we have

\begin{eqnarray}
(a',g')&=&(a,g).(b,h)\nonumber\\
       &=&(a+adj(g)(b),gh)
\end{eqnarray} 
which finishes the proof.
\end{proof}

\begin{definition}

We denote $E^*$ to be set of equivalence classes follows from (\ref{equi}), i.e. $$E^*=G^*\times_{H^*}TF$$ and we refer $E^*$ as the {\it{Prolonged Vector Bundle}}.
\end{definition}
\begin{remark} 
The bundle projection of the prolonged bundle is 
\begin{eqnarray}
\pi^*_{E^*}: G^* \times_{H^*} TF \to G^*/H^* \hspace{5cm}\nonumber \\
((a,g),v)H^* \to \pi^*_{E^*}(((a,g),v)H^*)=(a,g)H^*\hspace{15mm} \nonumber
\end{eqnarray}
and local trivialization of the prolonged bundle is 
\begin{eqnarray}
\psi^* : (\pi^*_{E^*})^{-1}(V) \to V \times TF\hspace{55mm} \nonumber \\
((a,g),v)H^* \to \psi^*(((a,g),v)H^* )=((a,gH),\sigma^*(b,h)(v))
\end{eqnarray}
where  $V \subset T(G/H)$ is an open subset.
\end{remark}

So far, we have given the structures of the prolonged bundle. In the following, we define a new prolongation that is obtained by the prolongation of homogeneous vector bundles.
\newline

\begin{definition}
Let $\rho:G \to Aut(\Gamma(E))$ be the induced representation of the homogeneous vector bundle $E$. Then $\rho^* :TG \to Aut(\Gamma(E^*))$ is called the prolongation of the induced representation $\rho$, where $E^*$ denotes the prolongation of $E$.  
\end{definition}

\section{Conclusion and Future Work}\label{conc}
In this paper, we have introduced a study of prolongations of homogeneous vector bundles. We have used one-to-one correspondence of homogeneous vector bundles and finite dimensional Lie group representations, and we have defined prolongations of homogeneous vector bundles. We introduced the geometric structures of this new bundle, such as local trivialization of the bundle, equivalence classes and the bundle projection. We have defined the prolongations of induced representations which are infinite dimensional linear representation. In future, we plan to study on prolongations of infinite dimensional representations by using one-to-one correspondence of homogeneous vector bundles.

\bigskip

\end{document}